\documentclass{amsproc}

\usepackage{amsmath,amsthm, amscd, amssymb, amsfonts, mathrsfs}
\usepackage{stmaryrd}
\usepackage{enumitem}
\usepackage{faktor} 

\usepackage[noadjust]{cite}

\usepackage[all]{xy}
\usepackage{color}

\usepackage[T1]{fontenc}
\usepackage{mathtools}

 \usepackage{tikz}
\usetikzlibrary{patterns}
\usetikzlibrary{arrows}
\usetikzlibrary{positioning}
\usetikzlibrary{decorations.pathmorphing}

\usepackage{multicol}
\usepackage{enumitem}

\usepackage{cite}

\usepackage[plainpages=false,pdfpagelabels,hypertexnames=false]{hyperref}

\allowdisplaybreaks

\newcommand{\cG}{\mathcal{G}}

\newcommand{\cJ}{\mathcal{J}}
\newcommand{\cO}{\mathcal{O}}
\newcommand{\cS}{\mathcal{S}}

\newcommand{\cM}{\mathcal{M}}
\newcommand{\cC}{\mathcal{C}}

\newcommand{\cD}{\mathcal{D}}

\newcommand{\cU}{\mathcal{U}}

\renewcommand{\_}[1]{_{\left( #1 \right)}}
\renewcommand{\^}[1]{^{\left( #1 \right)}}

\newcommand{\ot}{{\otimes}}
\newcommand{\ku}{\Bbbk}

\newcommand\fInd{\mathsf{Ind}}

\newcommand\fL{\mathsf{L}}
\newcommand\ofM{\overline{\mathsf{M}}}
\newcommand\fM{\mathsf{M}}
\newcommand\fN{\mathsf{N}}

\newcommand\fP{\mathsf{P}}
\newcommand\fQ{\mathsf{Q}}

\newcommand\fW{\mathsf{W}}

\newcommand{\Z}{{\mathbb Z}}
\newcommand{\C}{{\mathbb C}}

\newcommand{\Sn}{{\mathbb S}}

\newcommand{\BV}{{\mathfrak B}}

\newcommand{\ydh}{{}^H_H\mathcal{YD}}
\newcommand{\ydk}{{}^K_K\mathcal{YD}}

\newcommand{\End}{\operatorname{End}}

\newcommand{\Inf}{\operatorname{Inf}}

\newcommand{\ch}{\operatorname{ch}}
\newcommand{\chgr}{\operatorname{ch}^{\bullet}}

\newcommand\Hom{\operatorname{Hom}}

\newcommand\e{\varepsilon}

\numberwithin{equation}{section}
\newtheorem{lema}{Lemma}[section]
\newtheorem{theorem}[lema]{Theorem}
\newtheorem{cor}[lema]{Corollary}

\newtheorem{prop}[lema]{Proposition}

\newtheorem{question}[lema]{Question}
\newtheorem{question-app}{Question}

\theoremstyle{definition}

\theoremstyle{remark}

\newcommand{\onabla}{\overline{\nabla}}

\newcommand{\oV}{\overline{V}}
\newcommand{\oR}{\overline{R}}

\newcommand{\bU}{\mathbf{U}}

\newcommand{\bJ}{\mathbf{J}}

\newcommand{\kuu}{\mathfrak{u}}

\makeatletter
\def\Tenint{\@ifnextchar_{\@Tenintsub}{\@Tenintnosub}}
\def\@Tenintsub_#1{\mathchoice{\mathbin{\underline{\mathop{\otimes}}}_{#1}}%
  {\underline{\otimes}_{#1}}{\underline{\otimes}_{#1}}{\underline{\otimes}^L_{#1}}}
\def\@Tenintnosub{\mathbin{\underline{\mathop{\otimes}}}}
\makeatother

\title{On Hopf algebras with triangular decomposition}

\author{Cristian Vay}
\address{Universidad Nacional de C\'ordoba, Facultad de Matem\'atica, Astronom\'ia, F\'isica y Computaci\'on, CIEM--CONICET, C\'ordoba, Rep\'ublica Argentina}
\email{vay@famaf.unc.edu.ar}

\thanks{Partially supported by CONICET, Secyt (UNC), FONCyT PICT 2016-3957 and MathAmSud project GR2HOPF}


\begin{document}

\begin{abstract}
In this survey, we first review some known results on the representation theory of algebras with triangular decomposition, including the classification of the simple modules. We then discuss a recipe to construct Hopf algebras with 
triangular decomposition. Finally, we extend to these Hopf algebras the main results of \cite{proj-bgg} regarding projective modules over Drinfeld doubles of bosonizations of Nichols algebras and groups.
\end{abstract}

\maketitle

\section{Introduction}


Over the last years there have been remarkable progress on the problem of classifying finite-dimensional Hopf algebras. New methods to construct families with different properties were introduced and a lot of examples are now at hand. This has motivated an increasing interest in their representation theory.
The family of pointed Hopf algebras over abelian groups is deeply related to the universal enveloping of semisimple Lie algebras. It is natural then to draw inspiration from Lie theory to study modules in this context. We would like to generalize, for instance, the following well-known facts in Lie theory:
\begin{itemize}
 \item[$(*)$] The simple $U(\mathfrak{g})$-modules (in a suitable category) are classified by their highest weights. 
 \item[$(**)$] The Weyl character formula for the simple modules.
\end{itemize}

Let us briefly mention some works extending these properties to different families of Hopf algebras. Lusztig~\cite{MR954661} has demonstrated that $(*)$ and $(**)$ hold for the {\it$q$-analogs} $U_{q}(\mathfrak{g})$, for any symmetrizable Kac--Moody algebra $\mathfrak{g}$. We can see the $q$-analogs $U_{q}(\mathfrak{g})$ as examples of a more general family of pointed Hopf algebras 
introduced by Andruskiewitsch--Schneider. In \cite{MR2732981}, Andruskiewitsch--Radford--Schneider have extended $(*)$ and $(**)$ to this family following the strategy of Lusztig. This family also includes the pointed Hopf algebras previously considered in \cite{MR2407847} (associated to Cartan matrices of finite type) and in \cite{MR2051731,MR2338616,MR2642565} (multi-parametric 
version of $U_{q}(\mathfrak{g})$).

The character formula for the {\it $q$-divided power algebra} $\cU_q(\mathfrak{g})$ (with $q$ a root of unity of certain order) is a consequence of several 
results. First, Kashiwara--Tanisaki proved a character formula for affine Lie algebras at a negative level. Then, Kazhdan--Lusztig (and Lusztig) constructed a functor between suitable categories of modules over the above algebras which translate the character formulas from one setting to the other. Moreover, Andersen--Jantzen--Soergel \cite{MR1272539} demonstrated that the character formulas for the 
small quantum group $\kuu_q(\mathfrak{g})$ and the restricted enveloping Lie algebra $U^{[p]}(\mathfrak{g})$ are the same for large primes (and these 
are independent of the characteristic). Using these formulas and the Steinberg's tensor product theorem, one can deduce the formulas for $\cU_q(\mathfrak{g})$, and for the corresponding semisimple connected and simply connected algebraic group in positive characteristic. The reader can find more details and precise references on the various character formulas in the fascinating survey by Jantzen \cite{MR2528469}.

The above Hopf algebras can be constructed via a Drinfeld double procedure, starting from a pointed Hopf algebra over an abelian group. It is natural then to consider other types of Drinfeld doubles. In \cite{MR2279242}, Krop--Radford prove that $(*)$ holds for the Drinfeld double of a finite-dimensional Hopf algebra $H$ such that its coradical $H_0$ is a Hopf subalgebra complemented by a nilpotent Hopf ideal. In this setting, the Drinfeld double of $H_0$ plays the role of the Cartan subalgebra and their simple modules are the corresponding weights (unlike Lie theory, they are not necessarily one-dimensional). Different proofs were given for $H$ being the bosonization of a finite-dimensional Nichols algebra over a finite abelian group by Heckenberger--Yamane \cite{MR2840165} and by Pogorelsky and the author \cite{PV2} for general finite groups. Andruskiewitsch--Angiono \cite{andruskiewitsch-angiono-basic} recently extended the proof of \cite{PV2} to bosonizations over any finite-dimensional Hopf algebra. In 
this great generality, a character formula is not 
expected and we should restrict ourselves, for instance, to certain families of Nichols algebras. We refer the reader to \cite{doi:10.1080/00927872.2017.1357726,MR2504492,PV2,MR1743667} to see how the simple modules look like in some explicit examples.


\

In this survey, we first review a general framework due to Bellamy--Thiel~\cite{Bellamy-Thiel-1} and Bonnaf\'e--Rouquier~\cite[\S 
F.2]{bonnafe-rouquier-1} which ensures that $(*)$ holds. Explicitly, this is guaranteed for any graded algebra with {\it triangular decomposition} (under some mild extra conditions). Moreover, the simple modules can be obtained as quotients of {\it standard (Verma) modules} and, under certain assumptions, this leads to the construction of a highest weight category. We then discuss a recipe to construct Hopf algebras with triangular decomposition. We basically rewrite the Drinfeld double construction in order to obtain directly the triangular decomposition. The small quantum groups and the Drinfeld doubles of bosonizations of Nichols algebras can be obtained  via this procedure. Finally, we extend to this wider class of Hopf algebras some results of \cite{proj-bgg}, regarding projective modules over Drinfeld doubles of bosonizations of finite-dimensional Nichols algebras over groups.  These results rely fundamentally on properties of the Nichols algebra.

\

Roughly speaking, the main idea that we would like to share with the reader is the following: if your favorite Hopf algebra admits a suitable triangular decomposition, then its simple modules are classified by their highest weights and you can construct them as quotients of standard modules. 
Thus, the major problem to address is to compute the weight decomposition of the simple modules. 

\

The paper is organized as follows. The general framework of algebras with decomposition is summarized in Section \ref{sec:a general framework}. We fix the notation and conventions concerning Hopf and Nichols algebras in Section \ref{sec:Preliminaries}. The recipe to construct Hopf algebras with triangular decomposition is explained in Section \ref{sec:a general construction} and the representation theory of them is studied in Section \ref{sec:ext de proj-bgg}.

\subsection*{Acknowledgements}

I thank Nicol\'as Andruskiewitsch and Dmitri Nikshych for the invitation to participate in the special session ``Hopf algebras and tensor categories'' at the Mathematical Congress of the Americas 2017. This work was carried out during a research stay at the Universit\'e Clermont Auvergne supported by CONICET. I thank Simon Riche for his warm hospitality, for the interesting discussions on character formulas and for his willingness to answer all my questions. I also thank Ulrich Thiel for the discussion on \cite{Bellamy-Thiel-1}. I am grateful to the referee and Gast\'on Garc\'ia for the careful reading of the manuscript and their useful comments.

\section{A general framework}\label{sec:a general framework}

Throughout this work, $\ku$ is a field and {\it graded} always means $\Z$-graded. Let $A$ be a graded algebra. We denote by ${}_A\cM$ the category of finite-dimensional left $A$-modules and by ${}_A\cG$ the category of graded finite-dimensional left $A$-modules with morphisms preserving the grading.

We say that a graded algebra $A=\oplus_{n\in\Z}A_n$ admits a {\it triangular decomposition} if there 
exist graded subalgebras $A^-$, $T$ and $A^+$ such that the multiplication
\begin{align}\label{eq:triangular decomposition}
\tag{td}A^-\ot T\ot A^+\longrightarrow A
\end{align}
gives a linear isomorphism and the following properties hold
  \begin{enumerate}[label=(td\arabic*)]
  \item $A^{\pm}\subseteq\oplus_{n\in\Z_{\pm}}A_n$ and $T\subseteq A_0$.
  \item $(A^{\pm})_0=\ku$.
  \item $B^{\pm}:=A^{\pm}T=TA^{\pm}$.
\end{enumerate}

We denote by $\Lambda$ a complete set of non-isomorphic finite-dimensional simple $T$-modules. The elements of $\Lambda$ are called {\it weights}. If $\lambda\in\Lambda$ is a $T$-submodule of an $A$-module $\fN$ such that $(B^+)_{>0}\cdot\lambda=0$ (resp. $(B^-)_{<0}\cdot\lambda=0$), we say that $\lambda$ is a {\it highest weight} (resp. {\it lowest weight}) of $\fN$. In such a case, $\lambda$ is a simple $B^+$-submodule (resp. $B^-$-submodule) isomorphic to $\Inf^{B^+}_T(\lambda)$ (resp. $\Inf^{B^-}_T(\lambda)$), where $\Inf^{B^{\pm}}_T:{}_{T}\cM\rightarrow{}_{B^{\pm}}\cM$ denotes the functor induced by the projection $B^{\pm}\twoheadrightarrow T$. If $\fN$ is also generated by $\lambda$, we say that $\fN$ is a {\it highest weight module} (resp. {\it lowest weight module}).

The {\it proper standard module} associated to $\lambda\in\Lambda$ is
\begin{align}\label{eq:proper standard modules}
\ofM(\lambda)=A\ot_{B^+}\Inf^{B^+}_T(\lambda)
\end{align}
We denote by $\fL(\lambda)$ the head of $\ofM(\lambda)$. Then $\ku\ot_{B^+}\Inf^{B^+}_T(\lambda)$ is a highest weight of $\ofM(\lambda)$ isomorphic to $\lambda$, and so is its image in $\fL(\lambda)$. Hence $\ofM(\lambda)$ and $\fL(\lambda)$ are highest weight modules.

\subsection{The finite-dimensional case}\label{subsec:finite case}
In this subsection, we assume that $A$ is a finite-dimensional algebra with triangular decomposition as above.
Therefore the simple and projective modules over $A$ admit a 
grading by \cite{MR659212}.

Let $P_T(\lambda)$ be the projective cover of $\lambda\in\Lambda$. The {\it standard module} associated to $\lambda\in\Lambda$ is
\begin{align}\label{eq:standard modules}
\fM(\lambda)=A\ot_{B^+}\Inf^{B^+}_T(P_T(\lambda)).
\end{align}
Notice that 
\begin{align}\label{eq:standard modules semisimple case}
\fM(\lambda)=\ofM(\lambda)\quad\mbox{if $T$ is semisimple.} 
\end{align}
The {\it proper costandard module} $\onabla(\lambda)$ and {\it costandard module} $\nabla(\lambda)$ associated to $\lambda$ are defined in a dual way.\footnote{In \cite{Bellamy-Thiel-1} the standard modules are denoted by $\Delta$ but, as it is usual in Hopf algebra theory, we keep  $\Delta$ to denote the comultiplication.}


\begin{theorem}[\cite{MR1136900,Bellamy-Thiel-1}]\label{prop:general facts}
Assume that $T$ is split, {\it i.e.} $\End_T(\lambda)=\ku$ for all $\lambda\in\Lambda$. Then:
 \begin{enumerate}[label=(\roman*)]
  \item The head $\fL(\lambda)$ of $\ofM(\lambda)$ is simple for all $\lambda\in\Lambda$.
  \item Every simple $A$-module is isomorphic to $\fL(\lambda)$ for a unique $\lambda\in\Lambda$.
  \item\label{item:standard filtration} The projective cover $\fP(\lambda)$ of $\fL(\lambda)$ admits a {\it standard filtration} for all $\lambda\in\Lambda$, {\it i.e.} there exists a chain of submodules 
$0=\fN_0\subset\fN_1\subset\cdots\subset\fN_n=\fP(\lambda)$ such that for each $i$, $\fN_i/\fN_{i-1}\simeq\fM(\lambda_i)$ for some $\lambda_i\in\Lambda$.
  \item {\it Brauer Reciprocity} $[\fP(\lambda):\fM(\mu)]=[\onabla(\mu):\fL(\lambda)]$ holds for all $\lambda,\mu\in\Lambda$, {\it i.e.} the number of subquotients isomorphic to 
$\fM(\mu)$ in a standard filtration of $\fP(\lambda)$ is equal to the number of composition factors of type $\fL(\lambda)$ of $\onabla(\mu)$. Moreover, this number is independent of the chosen filtration.
\end{enumerate} 
Dual properties are also satisfied by (proper) costandard modules and the injective hull of $\fL(\lambda)$. 
\qed
\end{theorem}

Thus, the category of modules over a finite-dimensional algebra with triangular decomposition has the flavor of a {\it highest weight category} \cite[Definition 3.1]{CPS}. However, if $A$ is a non-semisimple Hopf algebra\footnote{and in other important examples, cf. \cite[\S1]{Bellamy-Thiel-1}.}, ${}_{A}\cM$ is not a highest weight category. In fact, $A$ has infinite global dimension (because it is symmetric and non-semisimple) which contradicts \cite[Theorem 4.3 (a)]{PS}.

We also see in the examples ({\it e.g.} \cite[\S1]{proj-bgg}) that the axiom of highest weight categories which fails is the existence of a partial order $>$ in $\Lambda$ such that: if $\fM(\mu)$ occurs in a standard filtration of $\fP(\lambda)$, then $\mu>\lambda$; and $\fM(\lambda)$ occurs precisely once. However, Bellamy-Thiel~\cite{Bellamy-Thiel-1} realized that this trouble can be fixed by using the grading; that is working inside of the category of graded modules. 

We denote by $\fN[i]$ the shift of $\fN\in{}_A\cG$ (or ${}_T\cG$) by $i\in\Z$, that is
\begin{align}\label{eq:gr shift}
(\fN[i])_n=\fN_{n-i}
\end{align}
for all $n\in\Z$. Shifting commutes with taking projective covers and injective hulls \cite{MR659212}. Also, it is compatible with the (proper) standard modules and the (proper) costandard modules. In particular,
\begin{align*}
\ofM(\lambda[i])=\ofM(\lambda)[i],\quad\fM(\lambda[i])=\fM(\lambda)[i]\quad\mbox{and}\quad\fL(\lambda[i])=\fL(\lambda)[i]
\end{align*}
for all $\lambda\in\Lambda$ and $i\in\Z$.

Since $T$ is concentrated in degree $0$, the simple objects in ${}_T\cG$ are parametrized by $\Lambda\times\Z$. The simple object corresponding to $(\lambda,i)$ is just $\lambda[i]$, the simple $T$-module $\lambda$ concentrated in degree $i$. We can then define a partial order on $\Lambda\times\Z$ by letting $\lambda[i]<\mu[j]$ if and only if $i<j$. 

\begin{theorem}[\mbox{\cite[Theorem 1.1]{Bellamy-Thiel-1}}]\label{teo:highest weight dim fin}
If $T$ is semisimple, then ${}_A\cG$ is a highest weight category whose set of weights is $\Lambda\times\Z$ and the standard modules are $\fM(\lambda)[d]$ for all 
$\lambda[d]\in\Lambda\times\Z$. In particular, the properties (i)--(iv) above hold in the category ${}_A\cG$ for all $\lambda[d]\in\Lambda\times\Z$. \qed
\end{theorem}

We would like to stress that the simple modules are characterized as follows.

\begin{cor}\label{cor:highest weight of L}
Let $\lambda\in\Lambda$ and $i\in\Z$. Then $\fL(\lambda)[i]$ is the unique simple highest weight module with highest weight $\lambda[i]$. 
\qed
\end{cor}

\begin{cor}\label{cor:lowest weight of L}
Let $\lambda\in\Lambda$ and $i\in\Z$. Then there exist $\overline{\lambda}\in\Lambda$ and $l_\lambda\in\Z_{\leq0}$ such that $\overline{\lambda}[l_\lambda]$ is a lowest weight of $\fL(\lambda)$. Moreover, $\fL(\lambda)[i]$ is the unique simple lowest weight module with lowest weight $\overline{\lambda}[l_{\lambda}+i]$. 
\end{cor}

\begin{proof}
If $\fN$ is a graded vector space, we write $\fN'$ for the same vector space, but with the reversed grading $(\fN')_n=\fN_{-n}$ for all $n\in\Z$. Then $A'$ admits a triangular decomposition with subalgebras $T$ and $(A')^{\pm}=A^{\mp}$, and $\fL(\lambda)'$ is a simple graded $A'$-module. By the above corollary, there exist $\overline{\lambda}\in\Lambda$ and $l_\lambda\in\Z_{\leq0}$ such that $\fL(\lambda)'$ is the unique simple highest weight $A'$-module with highest weight $\overline{\lambda}[l_\lambda]$. With respect to the grading of $A$, $\overline{\lambda}$ is a lowest weight and hence the corollary follows.
\end{proof}

An object in a highest weight category is {\it tilting} if it admits both standard and costandard filtrations.

\begin{theorem}[\mbox{\cite[Theorem 5.1]{Bellamy-Thiel-1}}]
If $T$ is semisimple, then the tilting objects in ${}_A\cG$ are precisely the projective-injective objects. \qed
\end{theorem}

\subsubsection{Graded characters}

Recall that the {\it Grothendieck group} $K(\cC)$ of an abelian category $\mathcal{C}$ is the abelian group generated by the isomorphism classes $[X]$ of objects in $\cC$ subject to the relations $[X]=[Y]+[Z]$ if there exists a short exact sequence $Y\rightarrow  X\rightarrow Z$. 

The inclusion $T\rightarrow A$ induces a group homomorphism $\ch:K({}_A\cM)\rightarrow K({}_T\cM)$. This does not necessarily distinguish simple modules, see for instance \cite[\S1.1]{PV2}. Once again, we can fix this issue by working inside the graded category. 

We consider $K({}_A\cG)$ and $K({}_T\cG)$ as $\Z[t,t^{-1}]$-modules where $t$ acts by the shift $[1]$. As $T$ is concentrated in degree $0$, we have that $K({}_T\cG)\simeq K({}_T\cM)[t,t^{-1}]$ and we can define the {\it graded character} $\chgr:K({}_A\cG)\longrightarrow K({}_T\cM)[t,t^{-1}]$  by
\begin{align}\label{eq:chgr}
\chgr \fN=\sum_{i\in\Z}\ch \fN_i\, t^i
\end{align}
for every $\fN\in{}_A\cG$. The following is mainly a consequence of the fact that the simple modules $\fL(\lambda)$ are distinguished by their highest weights.

\begin{prop}[\mbox{\cite[Proposition 3.19]{Bellamy-Thiel-1}}]\label{prop:chgr L are basis}
$\chgr:K({}_A\cG)\longrightarrow K({}_T\cM)[t,t^{-1}]$ is an injective morphism of $\Z[t,t^{-1}]$-modules. \qed
\end{prop}

From now on, we identify $K({}_A\cG)$ with its image via $\chgr$. The next corollary is a direct consequence of the definition of Grothendieck group. It states that we can read the composition factors of a graded $A$-module from its graded character, see Lemma \ref{le:coeficientes de los pol}.

\begin{cor}\label{cor:Llambda basis}
The set $\{\chgr\fL(\lambda)\mid\lambda\in\Lambda\}$ is a $\Z[t,t^{-1}]$-basis of $K({}_A\cG)$. More explicitly, for every $\fN\in{}_A\cG$ there exist unique polynomials $p_{\fN,\fL(\lambda)}\in\Z[t,t^{-1}]$ such that 
\begin{align*}
\chgr\fN=\sum_{\lambda\in\Lambda} p_{\fN,\fL(\lambda)}\,\chgr\fL(\lambda).
\end{align*}
\qed
\end{cor}

Let ${}_A\cG_{proj}$ be the category of finite-dimensional graded projective modules over $A$. The {\it split Grothendieck group} $K({}_A\cG_{proj})$ is the group generated by the isomorphism classes 
$[X]$ of graded projective modules subject to the relations $[X]=[Y]+[Z]$ if $X\simeq Y\oplus Z$. 

Again, we may consider $K({}_A\cG_{proj})$ as a $\Z[t,t^{-1}]$-module with the action given by the shift $[1]$. Clearly, we have a group morphism $K({}_A\cG_{proj})\longrightarrow K({}_A\cG)$ and we can consider the graded character on projective modules. The next proposition allows us to identify $K({}_A\cG_{proj})$ with its image under $\chgr$.

\begin{prop}\label{prop:chgr fM base para proj}
If $T$ is semisimple, then $\chgr:K({}_A\cG_{proj})\longrightarrow K({}_T\cM)[t,t^{-1}]$ is an injective morphism of $\Z[t,t^{-1}]$-modules.
Moreover, the set $\{\chgr\fM(\lambda)\mid\lambda\in\Lambda\}$ is a $\Z[t,t^{-1}]$-basis of the image of $\chgr$,  {\it i.e.} for every $\fP\in{}_A\cG_{proj}$ there exist unique polynomials $p_{\fP,\fM(\lambda)}\in\Z[t,t^{-1}]$ such that 
\begin{align*}
\chgr\fP=\sum_{\lambda\in\Lambda} p_{\fP,\fM(\lambda)}\,\chgr\fM(\lambda).
\end{align*}
\end{prop}

\begin{proof}
Since projective modules admit standard filtrations in ${}_A\cG$ by Theorem \ref{teo:highest weight dim fin}, the graded characters of standard modules generate $\chgr(K({}_A\cG_{proj}))$. Since $T$ is semisimple, $\fM(\lambda)=\ofM(\lambda)$ for all $\lambda\in\Lambda$ and hence the standard modules are distinguished by their highest weights. Following the arguments in the proof of \cite[Proposition 3.19]{Bellamy-Thiel-1}, one can show that $\chgr$ is injective and the proposition follows.
\end{proof}

\subsection{The infinite-dimensional case}

Following Bonnaf\'e--Rouquier \cite[\S F.2]{bonnafe-rouquier-1}, we can also construct a highest weight category from the category of modules over an algebra $A$ with triangular decomposition without the assumption  that neither $A$ is finite-dimensio\-nal nor $T$ is semi-simple. This is a generalization of the BGG category $\cO$ of a semisimple Lie algebra.

Let $A$, $B^{\pm}$ and $T$ be algebras satisfying (td)--(td3). Let $\Lambda$ be a set of non-isomorphic finite-dimensional simple $T$-modules. We consider now the full subcategory $\cO_T$ of $T$-modules whose objects are finite directs sums of modules in $\Lambda$. This is an abelian subcategory. For all $\lambda\in\Lambda$, we assume that:
\begin{enumerate}[label=($\Lambda$\arabic*)]
\item $\End_T(\lambda)=\ku$,
\item $B^{+}_{n}\ot_T\lambda$ and $B^{-}_{-n}\ot_T\lambda$ belong to $\cO_T$ and
\item\label{item:Lambda noether} the $A$-module $\ofM(\lambda)$ is noetherian ({\it e.g.} for $A^-$ noetherian).
\end{enumerate}

We still consider $\Lambda\times\Z$ partially ordered by the degrees and we say that a $B^+$-module $M$ is {\it locally nilpotent} if 
$M=\bigcup_{n\in\Z_+}\bigl\{m\in M\mid B^+_{>n}m=0\bigr\}$.

\begin{theorem}[\mbox{\cite[Theorem F.2.7 and Proposition F.1.17]{bonnafe-rouquier-1}}]\label{teo:bonnafe rouquier}
Let $\cO^{gr}$ be the category of finitely generated graded $A$-modules, locally nilpotent as $B^+$-modules and whose homogeneous components belong to $\cO_T$. Then 
\begin{enumerate}
 \item $\ofM(\lambda)[d]$ belongs to $\cO^{gr}$ and has a unique simple quotient $\fL(\lambda)[d]$ in $\cO^{gr}$ for all $\lambda[d]\in\Lambda\times\Z$. 
 \item Every simple object of $\cO^{gr}$ is isomorphic to $\fL(\lambda)[d]$ for a unique $\lambda[d]\in\Lambda\times\Z$.
 \item $\cO^{gr}$ is a highest weight category  whose set of weights is $\Lambda\times\Z$ and the standard modules are $\ofM(\lambda)[d]$ for all $\lambda[d]\in\Lambda\times\Z$. 
 \end{enumerate}
\end{theorem}


Similar to the finite-dimensional case, we can consider the graded character of $\fN\in\cO^{gr}$ given by
\begin{align*}
\chgr \fN=\sum_{i\in\Z}\ch \fN_i\, t^i\in K(\cO_T)[[t,t^{-1}]].
\end{align*}
These are now Laurent formal series.

\subsubsection{Remarks}

In the infinite-dimensional case, there may exist simple quotients of the standard modules that do not admit a grading. For instance, let $A$ be a polynomial algebra $T[x,y]$ for some algebra $T$ with a one-dimensional representation $\e:T\rightarrow\ku$. If we set $-\deg(x)=1=\deg(y)$, then $\ku[x]\ot T\ot\ku[y]\rightarrow A$ is a triangular decomposition. Hence $\fL_a=\ofM(\e)/\langle a-x\rangle$ is a one-dimensional simple $A$-module for all $a\in\ku$ but $\fL_a$ does not belong to the category $\cO^{gr}$.

Using this algebra, we can also see that not every simple $A$-module is locally nilpotent over $B^+$. Indeed, $\ku_a=T[y]/\langle a-y\rangle$ is not locally nilpotent over $B^+$ and consequently any non-trivial simple quotient of $\fM(\ku_a)$ neither.

%
%
%

\section{Preliminaries on Hopf algebras}\label{sec:Preliminaries}

We fix here the notation and conventions on Hopf algebras for the rest of the paper. 

Let $H$ be a Hopf algebra. We denote by $\Delta_H$, $\e_H$ and by $\cS_H$ the comultiplication, counit and antipode of $H$, respectively. We omit the subscript when it is clear from the context. We always assume that the antipode is invertible. We use Sweedler notation for the comultiplication and coactions, {\it e.g.} $\Delta(h)=h\_{1}\ot h\_{2}$.

Given a vector space $M$, we set $M^*=\Hom(M,\ku)$ and $\langle f,m\rangle$ denotes the evaluation of $f\in M^*$ in $m\in M$. If $M$ is an $H$-module, then so is $M^*$ via $\langle h\cdot f,m\rangle=\langle f,\cS(h)\cdot m\rangle$ for all $h\in H$, $f\in M^*$, $m\in M$.

\

We recall that the Grothendieck group $K({}_H\cM)$ is indeed a ring with unit $\e$, {\it i.e.} the one-dimensional simple module given by the counit, and multiplication induced by the tensor product. Moreover, if $H$ is a graded algebra, then $K({}_H\cG)$ is a $\Z[t,t^{-1}]$-algebra via the assignment $\e[\pm1]\mapsto t^{\pm1}$. Given two simple $H$-modules $\lambda$ and $\mu$, we denote by $\lambda\mu$ both the $H$-module $\lambda\ot\mu$ as well as its representative in $K({}_H\cM)$ (the meaning will be clear from the context).

\subsection{Nichols algebras}\label{subsec:nichols algebras} 

Our main references here are \cite{MR1714540, MR1913436}. Let $\ydh$ be the category of Yetter-Drinfeld modules over $H$. This a braided category with braiding
$$
V\ot W\rightarrow W\ot V,\quad v\ot w\mapsto v\_{-1}\cdot w\ot v\_{0}
$$
for all objects $V,W\in\ydh$, $v\in V$ and $w\in W$.

Let $V\in\ydh$. The tensor algebra $T(V)$ is a graded braided Hopf algebra with comultiplication induced by $\Delta(v)=v\ot1+1\ot v$ for all 
$v\in V$. The Nichols algebra of $V$ is 
$$
\BV(V)=T(V)/\cJ(V)
$$
where $\cJ(V)$ is the largest Hopf ideal of $T(V)$ generated as an ideal by homogeneous elements of degree $\geq2$. 

Let $B$ be a braided Hopf algebra in $\ydh$, {\it e.g.} $T(V)$ or $\BV(V)$. The bosonization $B\#H$ is the Hopf algebra with underlying vector space $B\ot H$, and multiplication and comultiplication given by
\begin{align*}
(x\# h)(\tilde{x}\#\tilde{h})=&x(h\_{1}\cdot \tilde{x})\#h\_{2}\tilde{h}\quad\mbox{and}\\
\Delta(x\# h)=&x\^{1}\#(x\^{2})\_{-1}h\_{1}\ot (x\^{2})\_{0}\#h\_{2}
\end{align*}
for all $x,\tilde{x}\in B$ and $h,\tilde{h}\in H$; with $\Delta_B(x)=x\^{1}\ot x\^{2}$.

If $\BV(V)$ is finite-dimensional, we denote by $n_{top}$ its maximum degree. We set
\begin{align}\label{eq:lambda V}
\lambda_V=\BV^{n_{top}}(V),
\end{align}
the homogeneous component of degree $n_{top}$ of $\BV(V)$. It is well-known that $\lambda_V$ is one-dimensional. Therefore $\lambda_V$ is a simple $H$-module and $H$-comodule. We can fix a generator $x_{top}\in\lambda_V$. Thus, there exists a 
group-like element $g_{top}\in H$ such that the coaction satisfies $(x_{top})\_{-1}\ot(x_{top})\_{0}=g_{top}\ot x_{top}$.

%

\subsection{Quasitriangular Hopf algebras}\label{subsec:QT} 
Here we refer to \cite{MR1220770}. Let $(H,R)$ be a quasitriangular Hopf algebra, {\it i.e.} $H$ is a Hopf algebra and $R=R\^{1}\ot R\^{2}\in H\ot H$ satisfies that
\begin{align*}
\Delta(R\^{1})\ot R\^{2}&=R\^{1}\ot r\^{1}\ot R\^{2}r\^{2},\quad \e(R\^{1}) R\^{2}=1,\\
\noalign{\smallskip}
R\^{1}\ot \Delta(R\^{2})&=R\^{1}r\^{1}\ot r\^{2}\ot R\^{2},\quad R\^{1}\e(R\^{2})=1\quad\mbox{and}\\
\noalign{\smallskip}
\Delta^{cop}(h)R&=R\Delta(h),\quad\forall h\in H;
\end{align*}
where $r$ denotes another copy of $R$. The element $R$ is called the {\it $R$-matrix} of $H$. It turns out that $R$ is invertible and $R^{-1}=\cS(R\^{1})\ot R\^{2}= R\^{1}\ot\cS^{-1}(R\^{2})$. 

Let $R\_{l}=\{R\^{1}f(R\^{2})\mid f\in H^*\}$ and $R\_{r}=\{f(R\^{1})R\^{2}\mid f\in H^*\}$. Then $R\_{l}$ and $R\_{r}$ are finite-dimensional Hopf subalgebras of $H$ and 
$(R\_{l})^{*cop}\longrightarrow R\_{r}$, $p\mapsto p(R\^{1})R\^{2}$, is an isomorphism of Hopf algebras \cite[Proposition 2]{MR1220770}. Let $\psi:R\_{r}\longrightarrow (R\_{l})^{*cop}$ be 
its inverse map. We define the pairing
\begin{align}\label{eq:form on H}
R\_{l}\ot R\_{r}\longrightarrow\ku,\quad a\ot b\mapsto\langle a,b\rangle=\langle a,\psi(b)\rangle
\end{align}
for all $a\in R\_{l}$ and $b\in R\_{r}$. It satisfies that
\begin{align}\label{eq:prop pairing}
\langle aa',b\rangle=\langle a,b_2\rangle\langle a',b_1\rangle,\quad\langle a,bb'\rangle=\langle a_1,b\rangle\langle a_2,b'\rangle,\quad\langle a,\cS(b)\rangle=\langle\cS^{-1}(a),b\rangle
\end{align}
for all $a,a'\in R\_{l}$ and $b,b'\in R\_{r}$. Moreover, it holds that 
\begin{align}\label{eq:relacion del doble en una QT}
ab=\langle a_1,b_1\rangle\langle a_3,\cS(b_3)\rangle b_2a_2
\end{align}
for all $a\in R\_{l}$ and $b\in R\_{r}$ by \cite[Theorem 2]{MR1220770}.



\

The category of $H$-modules is braided with braiding
\begin{align*}
c=\tau\circ R\quad\mbox{and}\quad c^{-1}=R^{-1}\circ\tau
\end{align*}
where $\tau$ is the usual twist map $\tau(a\ot b)=b\ot a$.

We can endow any $H$-module $M$ with the coaction $\lambda_c(x)=R\^{2}\ot R\^{1}x$,
for all $x\in M$. This gives a braided functor $\eta_c:({}_{H}\cM,c)\rightarrow\ydh$. Instead, if we endow $M$ with the coaction  $\lambda_{c^{-1}}(x)=\cS(R\^{1})\ot R\^{2}x$,
for all $x\in M$, we obtain a braided functor $\eta_{c^{-1}}:({}_{H}\cM,c^{-1})\rightarrow\ydh$.

Nichols algebras can be defined as above for every object in a braided category. Thus, given $M\in{}_{H}\cM$, we can consider the 
Nichols algebras $\BV(M,c)\in ({}_{H}\cM,c)$ and $\BV(M,c^{-1})\in ({}_{H}\cM,c^{-1})$. Moreover, these coincide with the Nichols algebras of $\eta_c(M)$ and $\eta_{c^{-1}}(M)$, 
respectively. This holds because the ideal of relations of a Nichols algebra can be identified with the kernel of the quantum symmetrizer which depends only on the braiding, cf. \cite[Proposition 3.2.12]{MR1714540} or \cite[Proposition 2.11]{MR1913436}.

\subsection{Drinfeld double} Let $K$ be a finite-dimensional Hopf algebra. Following \cite[Theorem 7.1.1]{majid-q}, the Drinfeld double $\cD(K)$ of $K$ is a Hopf algebra which is 
$K\ot K^*$ as a coalgebra. The multiplication and the antipode are given by
\begin{align}\label{eq:DH}
\begin{split}
(a\ot f)(a'\ot f')=&\langle f\_{1}, a'\_{1}\rangle\langle f\_{3},\cS_K(a'\_{3})\rangle (aa'\_{2}\ot f'f\_{2}),\\
\noalign{\smallskip}
\cS(a\ot f)=&(1\ot\cS_{K^*}^{-1}(f))(\cS_K(a)\ot\e),\quad\mbox{for all $a,a'\in K$ and $f,f'\in K^*$}.
\end{split}
\end{align}
It holds that $K$ and $K^{*op}$ are Hopf subalgebras of $\cD(K)$.

The Drinfeld double $\cD(K)$ is quasitriangular with $R$-matrix given by
\begin{align}\label{eq:Rmatrix de DH}
R=\sum_if_i\ot a_i,
\end{align}
where $\{a_i\}_i$ and $\{f_i\}_i$ are dual basis of $K$ and $K^*$, respectively.

There exists an equivalence of braided categories $\ydk\rightarrow({}_{\cD(K)}\cM,c)$ defined as follows. Given $M\in\ydk$, then $M\in{}_{\cD(K)}\cM$ with
\begin{align}\label{eq:ydk a DKcM}
(a\ot f)\cdot m=\langle f,m\_{-1}\rangle\, (a\cdot m\_{0})
\end{align}
for all $a\in K$, $f\in K^*$ and $m\in M$.

\section{A method to construct Hopf algebras with triangular decomposition}\label{sec:a general construction}

The Drinfeld double of bosonizations provides examples of Hopf algebras with triangular decomposition. However, we have to work a bit in order to identify the three subalgebras 
providing such a decomposition. We prefer instead to start at the end, {\it i.e.} we will consider three suitable algebras which lead to a Hopf algebra with triangular decomposition. 
These kind of constructions were performed for instance in \cite{MR1645545,MR3503231,MR2596372} in much more generality than we do it here, see \S\ref{sub:double 
bosonization}.

\

From now on, we fix a quasitriangular Hopf algebra $(H,R)$ and a finite-dimensio\-nal $H$-module $V$. We consider $V$ in $\ydh$ via the braided functor 
$({}_{H}\cM,c)\rightarrow\ydh$, {\it i.e.} the $H$-coaction of $x\in V$ is given by
\begin{align}\label{eq:coaction V}
x\_{-1}\ot x\_{0}=R\^{2}\ot R\^{1}x.
\end{align}

We let $\oV=V^*$ be the Yetter-Drinfeld module over $H$ with action and coaction given by
\begin{align}\label{eq:action Vdual}
\langle h y,x\rangle&=\langle y,\cS(h) x\rangle,\\
\label{eq:coaction Vdual}
y\_{-1}\ot y\_{0}&=\cS(R\^{1})\ot R\^{2}y
\end{align}
for all $y\in\oV$, $x\in V$ and $h\in H$. This corresponds to considering first the dual object of $V$ in ${}_{H}\cM$ and then applying the braided functor 
$({}_{H}\cM,c^{-1})\rightarrow\ydh$. It is immediate that 
\begin{align}\label{eq:relacion entre las coacciones}
\langle x\_{-1}\cdot y,x\_{0}\rangle=\langle y\_{0},\cS^{-1}(y\_{-1})\cdot x\rangle
\end{align}
for all $x\in V$ and $y\in\oV$. 


We then form the bosonization $T(V\oplus\oV)\#H$ and introduce the elements
\begin{align}\label{eq:el de U}
\llbracket y,x\rrbracket:=yx-(y\_{-1}\cdot x)y\_{0}&-\langle y,x\rangle+y\_{-1}x\_{-1}\langle y\_{0},x\_{0}\rangle
\end{align}
for all $y\in\oV$ and $x\in V$.
The  motivation for this formula is Example \S\ref{subsub:double of bosonization} below. 

Let $\bJ$ be the ideal of $T(V\oplus\oV)\#H$ generated by the elements $\llbracket y,x\rrbracket$ for all $x\in V$ and $y\in\oV$. We define
\begin{align}\label{eq:def U}
\bU_{(H,R)}(V)=\faktor{T(V\oplus\oV)\#H}{\bJ}.
\end{align}

\begin{prop}\label{prop:U}
$\bU_{(H,R)}(V)$ is a graded Hopf algebra with
\begin{align*}
\deg V=-1,\quad  \deg H=0,\quad \deg\oV=1
\end{align*}
and $T(V)\ot H \ot T(\oV)\longrightarrow\bU_{(H,R)}(V)$ is a triangular decomposition. Also, the Hopf subalgebra generated by $V$ and $H$, resp. $\oV$ and $H$, is isomorphic to
\begin{align*}
T(V)\#H,\quad\mbox{resp. $T(\oV)\#H$.} 
\end{align*}
\end{prop}

\begin{proof}
Clearly, $T(V\oplus\oV)\#H$ is graded and $\bJ$ is an homogeneous ideal. Moreover, it is a coideal because
\begin{align*}
\Delta(\llbracket y,x\rrbracket)=\llbracket y,x\rrbracket\ot1+y\_{-1}x\_{-1}\ot\llbracket y\_{0},x\_{0}\rrbracket
\end{align*}
for all $x\in V$ and $y\in\oV$. This follows from a straightforward computation using \eqref{eq:relacion del doble en una QT} and the identity $yx\_{-1}\ot x\_{0}=y\_{-3}x\_{-1}\cS(y\_{-1})y\_{0}\ot y\_{-2}\cdot x\_{0}$, which is a consequence of the axioms of the $R$-matrix. Therefore $\bU_{(H,R)}(V)$ is a graded Hopf algebra. 

The triangular decomposition holds because \eqref{eq:el de U} ensures that $yx\in V\oV\oplus H$ in $\bU_{(H,R)}(V)$, for all $y\in\oV$ and $x\in V$.

The last part of the statement is clear.
\end{proof}

We are interested in graded quotient Hopf algebras of $\bU_{(H,R)}(V)$ admitting a triangular decomposition with $H$ in the middle. The main example is the following. Let $\BV(V)$ and $\BV(\oV)$ be the Nichols algebras of $V$ and $\oV$, with defining ideals $\cJ(V)$ and $\cJ(\oV)$, respectively. We define
\begin{align}\label{eq:def u}
\kuu_{(H,R)}(V)=\faktor{\bU_{(H,R)}(V)}{\langle\cJ(V),\cJ(\oV)\rangle}.
\end{align}

\begin{prop}\label{prop:u}
$\kuu_{(H,R)}(V)$ is a graded Hopf algebra quotient of $\bU_{(H,R)}(V)$ and $\BV(V)\ot H \ot\BV(\oV)\longrightarrow\kuu_{(H,R)}(V)$ is a triangular decomposition.
Also, the Hopf subalgebra generated by $V$ and $H$, resp. $\oV$ and $H$, is isomorphic to
\begin{align*}
\BV(V)\#H,\quad\mbox{resp. $\BV(\oV)\#H$.} 
\end{align*}
\end{prop}

\begin{proof}
We only have to prove the triangular decomposition. A direct computation shows that the composition of the braidings $c_{V,\oV}:V\ot\oV\rightarrow\oV\ot V$ and $c_{\oV,V}:\oV\ot V\rightarrow V\ot\oV$ is the identity. Then, the Nichols algebra $\BV(V\oplus\oV)$ is isomorphic to $\BV(V)\ot\BV(\oV)$ as vector  spaces by \cite[Theorem 2.2]{MR1779599} and hence $\cJ(V\oplus\oV)=\cJ(V)\ot T(\oV)+T(V)\ot \cJ(\oV)=\langle\cJ(V),\cJ(\oV)\rangle$ as ideals in $T(V\oplus\oV)$. Therefore $\langle\cJ(V),\cJ(\oV)\rangle=\cJ(V)\ot H\ot T(\oV)+T(V)\ot H\ot\cJ(\oV)$ as ideals in $T(V\oplus\oV)\#H$. Since the generators of $\bJ$ do not belong to this ideal, we have that
\begin{align*}
\kuu_{(H,R)}(V)\simeq\faktor{\left(\faktor{T(V\oplus\oV)\#H}{\langle\cJ(V),\cJ(\oV)\rangle}\right)}{\bJ}\simeq\BV(V)\ot H\ot\BV(\oV)
\end{align*}
as vector spaces.
\end{proof}

%

\subsection{Examples}

\subsubsection{The Drinfeld double of a bosonization}\label{subsub:double of bosonization}

Let $K$ be a finite-dimensional Hopf algebra and $R$ the $R$-matrix of $\cD(K)$ given by \eqref{eq:Rmatrix de DH}. Let $V$ be a finite-dimensional Yetter-Drinfeld module over $K$ with $\dim\BV(V)<\infty$. We consider $V$ as a $\cD(K)$-module via \eqref{eq:ydk a DKcM}.

\begin{lema}
The Drinfeld double of $\BV(V)\#K$ is isomorphic to $\kuu_{(\cD(K),R)}(V)$.

In particular, $\BV(V)\ot\cD(K)\ot\BV(\oV)\longrightarrow\cD(\BV(V)\#K)$ is a triangular decomposition
and hence $\cD(K)$, $\BV(V)\#\cD(K)$ and $\BV(\oV)\#\cD(K)$ are graded Hopf subalgebras of $\cD(\BV(V)\#K)$.
\end{lema}


\begin{proof}
The isomorphism follows as in \cite[Lemma 7]{PV2}. We sketch the proof and leave the details for the reader. First, we have to check that 
$$(\BV(V)\#K)^{*op}\simeq\BV(\oV)\#K^{*op}$$
similar 
to \cite[Lemma 5]{PV2}. Here, we consider $\oV=V^*$ as the Yetter-Drinfeld module over $K^{*op}$ with action and coaction defined by
\begin{align*}
\langle f\cdot y,x\rangle=\langle f,\cS^{-1}(x\_{-1})\rangle\langle y,x\_{0}\rangle\quad\mbox{and}\quad
\langle y,h\cdot x\rangle=\langle y\_{-1},h\rangle\langle y\_{0},x\rangle
\end{align*}
for all $f\in K^{*op}$, $h\in H$, $x\in V$ and $y\in \oV$. This extends \cite[Definition 4]{PV2} given for $K=\ku G$. In order to prove the isomorphism, one has to consider $y\in\oV$ as an element of $(\BV(V)\#K)^{*op}$ by setting $\langle y,x\#h\rangle=\langle y,x\rangle\langle\varepsilon,x\rangle$
and $\langle y,b\rangle=0$
if $b$ is homogeneous of degree not equals to $1$.

Let us now identify $\BV(\oV)\#K^{*op}$ with a Hopf subalgebra of $\cD(\BV(V)\#K)$. Then, one can check that the Hopf algebra 
generated by $K$ and $K^{*op}$ is isomorphic to $\cD(K)$. Moreover $V$ and $\oV$ are 
invariants under the adjoint action of $\cD(K)$ and they are left ideals, {\it i.e.} left comodules under the comultiplication. More precisely, they are the Yetter-Drinfeld modules over 
$\cD(K)$ defined by \eqref{eq:coaction V}, \eqref{eq:action Vdual} and \eqref{eq:coaction Vdual}. 

To conclude, we claim that the elements of $V$ and $\oV$ satisfy the relation \eqref{eq:el de U}. In fact, if $x\in V$ and $y\in\oV$, it holds that
\begin{align*}
yx=(y\_{-1}\cdot x)y\_{0}+\langle y,x\rangle -\langle y\_{-2},x\_{-3}\rangle\langle y\_{0},\cS(x\_{-1})\cdot x\_{0}\rangle\, x\_{-2}y\_{-1}
\end{align*}
by \eqref{eq:DH}. On the other hand, the commutation rule \eqref{eq:relacion del doble en una QT} implies that
\begin{align*}
y\_{-1}x\_{-1}\langle y\_{0}, x\_{0}\rangle=\langle y\_{-2},x\_{-3}\rangle\langle y\_{0},\cS(x\_{-1})\cdot x\_{0}\rangle\, x\_{-2}y\_{-1}
\end{align*}
where $\langle y\_{-2},x\_{-3}\rangle$ denotes the pairing \eqref{eq:form on H}; compare with \cite[(21)]{PV2}.
\end{proof}

\subsubsection{The small quantum groups}

Let $\mathfrak{g}$ be a finite-dimensional semisimple Lie algebra over $\C$ and $q$ a root of unity of odd order. The small quantum group $\kuu_{q}(\mathfrak{g})$ can be constructed following our recipe. Although this is not the most elegant way of doing it, we describe it here for the sake of completeness.

Let $(c_{ij})_{1\leq i,j\le\theta}$ be the Cartan matrix of $\mathfrak{g}$ and $(d_i)_{1\leq i\le\theta}$ a diagonal matrix such that $(d_ic_{ij})_{1\leq i,j\le\theta}$ is 
symmetric. Let $\Gamma$ be the finite abelian group generated by $g_1, \dots g_\theta$ with $g_i^{\operatorname{ord}(q)}=1$ for all $1\leq i\le\theta$. Let $\chi_1, \dots 
\chi_\theta\in\widehat\Gamma$ with $\chi_{i}(g_j)=q_{ij}=q^{d_ic_{ij}}$ for all $1\leq i,j\le\theta$. 

If $g=g_1^{n_1}\cdots g_\theta^{n_\theta}\in\Gamma$, we denote $\chi_g=\chi_1^{n_1}\cdots \chi_\theta^{n_\theta}\in\widehat\Gamma$. Then $\ku\Gamma$ is quasitriangular with $R$-matrix
$$\oR=\frac{1}{|\Gamma|}\sum_{g,h\in\Gamma}\chi_g(h^{-1})g\ot h,$$
cf. \cite[Example 2.1.6]{majid-q}. This is the image of the $R$-matrix of the Drinfeld double $\cD(\ku\Gamma)$ under the epimorphism 
$\cD(\ku\Gamma)\twoheadrightarrow \ku\Gamma$ given by $\chi_g=g$ for all $g\in\Gamma$. 

Let $V=\ku\{x_1, \dots,x_\theta\}$ be the $\ku\Gamma$-module with action $g\cdot x_i=\chi_i(g)$ for all $1\leq i\le\theta$ and $g\in\Gamma$. If we apply to $V$ and $(\ku\Gamma,\oR)$ our 
construction, we obtain that
\begin{align*}
\kuu_{q}(\mathfrak{g})\simeq\kuu_{(\ku\Gamma,\oR)}(V)\simeq\faktor{\kuu_{(\cD(\ku\Gamma),R)}(V)}{\langle\chi_g-g\mid g\in\Gamma\rangle}
\end{align*}
where $R$ is the $R$-matrix of $\cD(\ku\Gamma)$, recall \eqref{eq:Rmatrix de DH}.

\subsection{More general constructions}\label{sub:double bosonization}

We have restricted ourselves to consider quasitriangular Hopf algebras and Nichols algebras because we are interested in the Drinfeld doubles of their bosonizations, and mainly those where $H$ is not an abelian group. However, this recipe is a particular case of more general constructions as we explain below.

\subsubsection{}
Let $(H,R)$ be a quasitriangular Hopf algebra. The double-bosonization construction of Majid \cite{MR1645545} can be performed for every dually paired braided Hopf algebras $B\in{}_{H}\cM$ and $C\in\cM_{H}$ instead of $T(V)$ and 
$T(\oV)$. Notice that $B$ and $C$ belong to different categories. We prefer in our exposition to consider $T(V)$ and $T(\oV)$ in the same category because this fact brings with it some 
advantages to study modules over the whole algebra $\bU_{(H,R)}(V)$ and their quotients. We will implicitly use this fact in \S\ref{sec:ext de proj-bgg}.

Also, we can then deduce that the action maps
$$
V\ot M\rightarrow M\quad\mbox{and}\quad\oV\ot M\rightarrow M
$$
are morphism of $H$-modules for all $M\in{}_{H}\cM$. This allows us to carry out an algorithm in order to compute submodules, cf. \cite[Remark 18]{PV2}.

\subsubsection{}

In \cite{MR3503231}, Laugwitz analyzes how to endow $T(V)\ot H\ot T(V^*)$ (and its quotients) with a structure of Hopf algebra. In particular, he extends the double-bosonization 
construction to a non-quasitriangular Hopf algebra $H$ \cite[\S3.4]{MR3503231}. The data for his construction are two braided Hopf algebras $B,C\in\ydh$ {\it weakly dually paired} 
\cite[Definition 6]{MR3503231}.

\subsubsection{} There is an important class of Hopf algebras with triangular decomposition which we cannot construct from our recipe. Namely, those with a free abelian group playing the 
role of $T$. These kind of Hopf algebras were studied for instance by Andruskiewitsch--Radford--Schneider\cite{MR2732981} and Heckenberger--Yamane \cite{MR2840165}, as we have mentioned in the introduction. The reader can find a systematic way to construct this sort of Drinfeld doubles in \cite[\S4]{MR2596372}.

\subsection{On the Nichols algebras of \texorpdfstring{$V$}{V}, \texorpdfstring{$\oV$}{oV} and \texorpdfstring{$V^*$}{Vd}}

Let $V^*$ denote the dual object of $V$ in $\ydh$. That is, $V^*=\oV$ as $H$-modules and the coaction of $y\in\oV$ is 
$$
y\_{-1}\ot y\_{0}=\sum_{i\in I}\cS^{-1}(R\^{2})\ot\langle y,R\^{1} x_i\rangle y_i=\sum_{i\in I}\cS^{-1}(R\^{2})\ot\cS^{-1}(R\^{1}) y,
$$ 
cf. \cite[Proposition 2.1.1]{MR1714540}. In the last equality we used \eqref{eq:action Vdual}. Notice that, $V^*$ is not necessarily isomorphic to $\oV$ as an $H$-comodule. For instance, if $H$ is the Drinfeld double of a non-abelian group $G$, then $\cS^{-1}(R\^{2})\in\ku G$ and $R\^{1}\in (\ku G)^*$, hence there is no isomorphism of comodules between these objects.

We now explain the relation between the Nichols algebras $\BV(V)$, $\BV(\oV)$ and $\BV(V^*)$. First, by \cite[Proposition 3.2.30]{MR1714540}, we know that 
\begin{align}\label{eq:iso BVVd and BVV dbop}
\BV(V^*)\simeq\BV(V)^{*bop}\quad\mbox{as braided Hopf algebras in $\ydh$.}
\end{align}
This isomorphism also holds in $({}_{H}\cM,c)$ via the forgetful functor $\ydh\rightarrow({}_{H}\cM,c)$, since it is a braided functor.

Since the comodule structures do not coincide, $\oV$ and $V^*$ are not equivalent as braided vector spaces. Instead, the braiding of $\oV\ot\oV\rightarrow\oV\ot\oV$ in $\ydh$ coincides 
with $c^{-1}:V^*\ot V^*\rightarrow V^*\ot V^*$ in ${}_{H}\cM$ by construction. Thus, we should analyze the Nichols algebra $\BV(\oV,c^{-1})$ in $({}_{H}\cM,c^{-1})$. 

By \cite[Lemma 1.11]{MR2766176}, we know that $\cJ(\oV)=\cJ(V^*)$ and
\begin{align}\label{eq:mult of BVoV}
\BV(\oV)=\BV(V^*)\quad\mbox{as $H$-module algebras.}
\end{align}
Moreover, we also see in the proof of \cite[Lemma 1.11]{MR2766176} that the comultiplication of $\BV(\oV)$ satisfies
\begin{align}\label{eq:comult of BVoV}
\Delta_{\BV(\oV)}=c^{-1}\circ\Delta_{\BV(V^*)},
\end{align}
that is the opposite coalgebra in the braided category $({}_H\cM,c)$. Therefore, by \eqref{eq:iso BVVd and BVV dbop} and \eqref{eq:mult of BVoV}, we have that 
\begin{align}\label{eq:isos para bgg}
\BV^n(\oV)\simeq\BV^n(V)^*\quad\mbox{as $H$-modules for all $n$.}
\end{align}
In particular, if $\dim\BV(V)<\infty$ and $n=n_{top}$, the above isomorphism implies that
\begin{align}\label{eq:dualidad lambda V lambda oV}
\lambda_V\lambda_{\oV}=\e=\lambda_{\oV}\lambda_{V}
\end{align}
in the Grothendieck ring $K({}_H\cM)$, recall \eqref{eq:lambda V}.

%

\section{On the representation theory of \texorpdfstring{$\kuu_{(H,R)}(V)$}{uHR(V)}}\label{sec:ext de proj-bgg}

Throughout this section, $(H,R)$ will be a quasitriangular Hopf algebra and $V$ a finite-dimensional $H$-module.  We can classify the simple $\kuu_{(H,R)}(V)$-modules using the results of \S\ref{sec:a general framework}. Theorems \ref{prop:general 
facts} and \ref{teo:highest weight dim fin} imply the following, see also \S\ref{subsec:on proj} below.

\begin{theorem}
Assume that $H$ and $\BV(V)$ are finite-dimensional and $H$ is split. Let $\Lambda$ be a complete set of non-isomorphic simple $H$-modules and $\fL(\lambda)$ be the head of $\ofM(\lambda)$. Then $\bigl\{\fL(\lambda)[d]\bigr\}_{\Lambda\times\Z}$ is a complete set of non-isomorphic simple graded $\kuu_{(H,R)}(V)$-modules. 

Moreover, if $H$ semisimple, the category of graded $\kuu_{(H,R)}(V)$-module is a highest weight category.
\qed 
\end{theorem}

We would like to stress that the graded character in the Hopf algebra setting is an algebra map. Explicitly, we can restate Proposition \ref{prop:chgr L are basis} as follows.

\begin{prop}
The functor ${}_{\kuu_{(H,R)}(V)}\cG\longrightarrow{}_{H}\cG$ given by restriction of scalars is monoidal and hence  $\chgr:K({}_{\kuu_{(H,R)}(V)}\cG)\longrightarrow K({}_H\cM)[t,t^{-1}]$ is an injective morphism of $\Z[t,t^{-1}]$-algebras.\qed
\end{prop}

In the infinite-dimensional case we apply Theorem \ref{teo:bonnafe rouquier}.

\begin{theorem}
Let $\Lambda$ be a set of non-isomorphic finite-dimensional simple $H$-modules and $\cO_H$ the full subcategory of $H$-modules whose objects are finite directs sums of modules in 
$\Lambda$. Assume  for all $\lambda\in\Lambda$ that 
\begin{enumerate}[label=($\Lambda$\arabic*)]
\item $\End_H(\lambda)=\ku$.
\item $\BV^n(V)\ot_H\lambda$ and $\BV^n(\oV)\ot_H\lambda$ belong to $\cO_H$.
\item the $\kuu_{(H,R)}(V)$-module $\ofM(\lambda)$ is noetherian ({\it e.g.} for $\BV(V)$ noetherian).
\end{enumerate}
Let $\cO^{gr}$ be the category of finitely generated graded $\kuu_{(H,R)}(V)$-modules, locally nilpotent as $\BV(\oV)$-modules and whose homogeneous components belong to $\cO_H$.

Then all the simple objects (up to isomorphism) in $\cO^{gr}$ are the heads $\fL(\lambda)[d]$ of $\ofM(\lambda)[d]$ for all $\lambda[d]\in\Lambda\times\Z$. Moreover, 
$\cO^{gr}$ is a highest weight category. \qed
\end{theorem}

\subsection{Some problems on the simple modules}

Once we have classified the simple modules, the question regarding their graded characters naturally arises.

\begin{question}
Describe $\chgr\fL(\lambda)$ for all $\lambda\in\Lambda$.
\end{question}

For instance, this was done for the Drinfeld double of the Taft algebra in \cite{MR1743667}; for the Drinfeld double of the Nichols algebra of unidentified diagonal type $\mathfrak{ufo}(7)$ in \cite{doi:10.1080/00927872.2017.1357726}; for the Drinfeld double of the Fomin-Kirillov algebra over $\Sn_3$ in \cite{PV2}; and for the small quantum group at a root of unity of large prime-order, a character formula can be deduced from \cite{MR1272539}. 

In examples we see that the Hilbert series of $\fL(\lambda)$ has symmetric coefficients, {\it i.e.} $\dim\fL(\lambda)_{n}=\dim\fL(\lambda)_{l_{\lambda}-n}$, where we assume that $\fL(\lambda)$ is finite-dimensional with minimum degree $l_{\lambda}$; recall  Corollary \ref{cor:lowest weight of L}. The Hilbert series of any finite-dimensional Nichols algebra also satisfies this symmetry. The following question for rational Cherednik algebras was posed by Thiel in \cite[Question 7.7(a)]{Thiel-CHAMP} and \cite[Problem 6.6]{Thiel-EMS}; nevertheless some restrictions are required because he also gives counter-examples in the first paper, see \cite[Remark 7.8]{Thiel-CHAMP}.

\begin{question}
Are the coefficients of the Hilbert series of $\fL(\lambda)$ symmetric?
\end{question}

Assume that $\BV(V)$ and $H$ are finite-dimensional. If $\lambda\in\Lambda$, we denote by $\overline{\lambda}$ the lowest weight of $\fL(\lambda)$; recall Corollary \ref{cor:lowest weight of L}. Then $\overline{\phantom{\lambda}}:\Lambda\to \Lambda$, $\lambda\mapsto\overline{\lambda}$, is a bijection. For the Drinfeld double of the Fomin-Kirillov algebra over $\Sn_3$, this bijection corresponds to the unique non-trivial braided autoequivalence of the category of $\cD(\ku\Sn_3)$-modules, cf. \cite[(2)]{PV3}, \cite[\S6.6]{LENTNER2017264} and \cite[\S 8.1]{NIKSHYCH2014191}.

\begin{question}
Does the bijection $\overline{\phantom{\lambda}}:\Lambda\to \Lambda$ induce a braided autoequivalence in the category of $H$-modules?
\end{question}

\subsection{On projective modules in the finite-dimensional case}\label{subsec:on proj}

From now on, we assume that $H$ and $\BV(V)$ are finite-dimensional and  $H$ is semisimple.

In this situation Theorem \ref{prop:general facts} gives us information on the projective $\kuu_{(H,R)}(V)$-modules:
\begin{itemize}
 \item Projective modules admit standard filtrations.
 \item Brauer Reciprocity holds.
\end{itemize}
We can also extend to $\kuu_{(H,R)}(V)$ some of the results in 
\cite{proj-bgg} where we have considered $H=\cD(G)$. Indeed, these results 
rely on the fact that the homogeneous component of maximum degree of a finite-dimensional Nichols algebra is one-dimensional, recall \S\ref{subsec:nichols algebras}. We give the main ideas of the proofs. The reader can complete the details following 
\cite{proj-bgg}.

We set
\begin{align*}
A=\kuu_{(H,R)}(V),\quad B^+=\BV(\oV)\#H\quad\mbox{and}\quad B^-=\BV(V)\#H,
\end{align*}
and keep the notation of \S\ref{sec:a general framework}. In particular, $\Lambda$ is a complete set of non-isomorphic simple $H$-modules. As $H$ is assumed to be semisimple, the proper standard modules in \eqref{eq:proper standard modules} and the standard modules in \eqref{eq:standard modules} coincide with the Verma modules considered in 
\cite{proj-bgg}, that is
\begin{align*}
\fM(\lambda)=A\ot_{B^+}\Inf^{B^+}_H(\lambda)
\end{align*}
for all $\lambda\in\Lambda$. 

Meanwhile the costandard modules 
$\nabla(\lambda)$ and the coVerma modules $\fW(\lambda)$ of \cite[(13)]{proj-bgg} are related by 
\begin{align*}
\fW(\lambda):=A\ot_{B^-}\Inf^{B^-}_H(\lambda)\simeq\nabla(\lambda_{\oV}\lambda)
\end{align*}
for all $\lambda\in\Lambda$. 


\begin{lema}[\mbox{\cite[Lemma 1]{proj-bgg}}]\label{le:verma is proj cover}
For all $\lambda\in\Lambda$, $\fM(\lambda)$ is the projective cover of $\Inf^{B^-}_H(\lambda)$ and the injective hull of $\Inf^{B^-}_H(\lambda_V\lambda)$ in ${}_{B^-}\cG$.
\end{lema}

\begin{proof}
As $B^-$-module, $\fM(\lambda)\simeq B^-\ot_H\lambda$ is an induced module from a semisimple algebra and hence it is projective and injective (because $B^-$ is Frobenius). It is indecomposable because its socle is simple and isomorphic to $\lambda_V\lambda$ (here we use that $\dim\lambda_V=1$).	
\end{proof}

\begin{lema}[\mbox{\cite[(10)]{proj-bgg}}]\label{le:verma dual}
For all $\lambda\in\Lambda$, $\fM(\lambda)^*\simeq\fM((\lambda_V\lambda)^*)$ as $A$-modules.
\end{lema}

\begin{proof}
 Using the grading, we see that $(\BV^{n_{top}}(V)\ot\lambda)^*\simeq(\lambda_V\ot\lambda)^*$ is a highest weight of $\fM(\lambda)^*$.
Then there is a morphism $f:\fM((\lambda_V\lambda)^*)\longrightarrow\fM(\lambda)^*$. If  $S$ denotes the socle of $\fM((\lambda_V\lambda)^*)$ as $B^-$-module, then $S$ is simple and $f(S)\neq0$ (cf. the proof of \cite[(10)]{proj-bgg}). Hence $f$ is injective and therefore an isomorphism because $\dim\fM(\lambda)^*=\dim\fM((\lambda_V\lambda)^*)$.
\end{proof}

Given an $H$-module $S$, we set $\fInd(S):=A\ot_H S$.

\begin{lema}[\mbox{\cite[Lemma 4]{proj-bgg}}]\label{le: W ot M}
For all $\lambda,\mu\in\Lambda$, $\fW(\lambda)\ot\fM(\mu)\simeq\fInd\bigl(\lambda\mu\bigr)$.
\end{lema}

\begin{proof}
The $H$-submodule $1\ot\lambda\ot1\ot\mu$ of $\fW(\lambda)\ot\fM(\mu)$ induces a morphism of $A$-modules $f:\fInd\bigl(\lambda\mu\bigr)\rightarrow\fW(\lambda)\ot\fM(\mu)$. Using that $\lambda_V$ is 
one-dimensional, we can see that $f$ is injective and then it is an isomorphism.
\end{proof}

Let $\fN\in{}_A\cG$ and $\fP\in{}_A\cG_{proj}$. By Corollary \ref{cor:Llambda basis} and Proposition \ref{prop:chgr fM base para proj}, there exist polynomials
\begin{align*}
p_{\fN,\fL(\lambda)}=\sum_ia_{\fN,\fL(\lambda),i}\,t^i\in\Z[t,t^{-1}]\quad\mbox{and}\quad
p_{\fP,\fM(\lambda)}=\sum_ia_{\fP,\fM(\lambda),i}\,t^i\in\Z[t,t^{-1}]
\end{align*}
such that
\begin{align*}
\chgr\fN=\sum_{\lambda\in\Lambda} p_{\fN,\fL(\lambda)}\,\chgr\fL(\lambda)\quad\mbox{and}\quad
\chgr\fP=\sum_{\lambda\in\Lambda} p_{\fP,\fM(\lambda)}\,\chgr\fM(\lambda).
\end{align*}

\begin{lema}\label{le:coeficientes de los pol}
The coefficients of $p_{\fN,\fL(\lambda)}$ and $p_{\fP,\fM(\lambda)}$ are given by
\begin{align*}
a_{\fN,\fL(\lambda),i}&=\#\{\mbox{composition factors of $\fN$ isomorphic to $\fL(\lambda)[i]$}\}\quad\mbox{and}\\
a_{\fP,\fM(\lambda),i}&=\#\{\mbox{subquotients isomorphic to $\fM(\lambda)[i]$ in a standard filtration of $\fP$}\}.
\end{align*}
\end{lema}

\begin{proof}
The first equality follows by induction on the length of a composition series of $\fN$. We use that the right hand side is  $\dim\Hom_{{}_A\cG}(\fP(\lambda)[i],\fN)$ and the exactness of $\Hom_{{}_A\cG}(\fP(\lambda)[i],-)$.

For the second one, we note that a standard filtration of $\fP$ induces a decomposition of $\fP$ (as $B^-$-module) into the direct sum of standard modules. Since the socle of $\fM(\lambda)$ (as $B^-$-module) is isomorphic to $\lambda_{V}\lambda$ by Lemma \ref{le:verma is proj cover}, the right hand side is $\dim\Hom_{{}_{B^-}\cG}(\lambda_V\lambda[i-n_{top}],\fP)$.
\end{proof}

The following is a graded version of the well-known BGG Reciprocity. In the general framework of \S\ref{subsec:finite case} this holds if $B^+$ and $B^-$ are related via an isomorphism similar to \eqref{eq:isos para bgg}, see \cite[Theorem 1.3]{Bellamy-Thiel-1}, and it is a consequence of Brauer Reciprocity. 

Given a polynomial $p(t,t^{-1})\in\Z[t,t^{-1}]$, we set $\overline{p(t,t^{-1})}=p(t^{-1},t)$. Since $A$ is Frobenius, for $\mu\in\Lambda$ there exist $\widehat{\mu}\in\Lambda$ and $s_{\widehat{\mu}}\in\Z$ 
such that $\fL(\widehat{\mu})[s_{\widehat{\mu}}]$ is the socle of $\fP(\mu)$ . Moreover, $\fP(\mu)$ is the injective hull 
of $\fL(\widehat{\mu})[s_{\widehat{\mu}}]$.

\begin{theorem}[\mbox{\cite[Corollary 12]{proj-bgg}}]\label{teo:BGG}
Let $\lambda,\mu\in\Lambda$. Then
\begin{align*}
p_{\fP(\mu),\fM(\lambda)}=t^{s_{\widehat{\mu}}}\,\overline{p_{\fM(\lambda),\fL(\widehat{\mu})}}.
\end{align*} 
In particular, $\widehat{\mu}=\mu$, if $A$ is symmetric, and $s_{\widehat{\mu}}=0$, if $A$ is graded symmetric. These properties hold if $A=\kuu_{(\cD(K),R)}(V)$ is 
a Drinfeld double as in \S\ref{subsub:double of bosonization}. 
\end{theorem}

\begin{proof}
Let $\overline{\lambda}[l_\lambda]$ be the lowest weight of $\fL(\lambda)$. Then $\fL(\lambda)^*\simeq\fL(\overline{\lambda}^*)[-l_\lambda]$ as in \cite[(26)]{proj-bgg}. We then deduce that $\fP(\lambda)^*\simeq\fP(\overline{\widehat{\lambda}}^*)[-l_{\widehat{\lambda}}-s_{\widehat{\lambda}}]$ as \cite[Lemma 8 (vi)]{proj-bgg}. Notice that \cite[Theorem 10]{proj-bgg} also holds in this context because it relies on \eqref{eq:isos para bgg}. Therefore the equality in the statement follows by combining \cite[Theorem 10]{proj-bgg} with the linear isomorphisms
\begin{align*}
\Hom_{{}_A\cG}(\fP(\mu)[i],\fM(\lambda))&\simeq\Hom_{{}_A\cG}(\fM(\lambda)^*,\fP(\mu)[i]^*)\quad\mbox{and}\\
\Hom_{{}_{B^-}\cG}(\lambda_V\lambda[i-n_{top}],\fP(\mu))&\simeq\Hom_{{}_A\cG}(\fP(\overline{\widehat{\lambda}}^*)[i-n_{top}-l_{\widehat{\lambda}}-s_{\widehat{\lambda}}],\fW(\lambda_V\lambda)^*),
\end{align*}
see the proof of \cite[Corollary 11 and 12]{proj-bgg}.

If $A$ is symmetric, then the head and the socle of $\fP(\mu)$ are isomorphic as $H$-modules, that is $\widehat{\mu}=\mu$. The definition of graded symmetric is that these modules are isomorphic as graded $H$-modules, that is $s_{\widehat{\mu}}=0$.

If $A=\kuu_{(\cD(K),R)}(V)$, then $A$ is symmetric by \cite{MR0347838,MR1220770}, cf.  \cite[p. 488, (3)]{MR1435369}. Moreover, $A$ is graded symmetric by the same proof of \cite[Lemma 3.1 (i)]{proj-bgg}.
\end{proof}

If we evaluate the polynomials of the statement at $t=1$, we obtain the BGG Reciprocity:
\begin{align*}
[\fP(\mu):\fM(\lambda)]=[\fM(\lambda):\fL(\widehat{\mu})]
\end{align*}
for all $\lambda,\mu\in\Lambda$. The next equivalence is a direct consequence of this equality.

\begin{cor}
A standard module is projective if and only if it is simple.\qed 
\end{cor}

Proposition \ref{prop:chgr fM base para proj} and Lemmas \ref{le:verma is proj cover} and \ref{le:verma dual} can be restated for $\fW(\lambda)$ instead of $\fM(\lambda)$. In particular, the projective modules admit costandard filtrations and there exist polynomials $p_{\fP,\fW(\lambda)}$ such that $\chgr\fP=\sum_{\lambda\in\Lambda}p_{\fP,\fW(\lambda)}\chgr\fW(\lambda)$. Thus, the proof of the next result is equal to that of \cite[Theorem 21]{proj-bgg} using Lemma \ref{le: W ot M}.

\begin{theorem}[\mbox{\cite[Theorem 21]{proj-bgg}}]\label{teo:tensor prod of proj}
Let $\fP$ and $\fQ$ be projective in ${}_A\cG$. Then
\begin{align*}
\fP\ot\fQ\simeq\oplus_{\lambda,\mu\in\Lambda}\,p_{\fP,\fW(\lambda)}\,p_{\fQ,\fM(\mu)}\,\fInd(\lambda\mu).
\end{align*}
\qed
\end{theorem}

\subsection{Rigid modules}

In the general setting of \S\ref{sec:a general framework}, a simple module satisfying $\fL(\lambda)\simeq\lambda$ is called {\it rigid} \cite[\S3.8]{Bellamy-Thiel-1}. These modules play an 
important role in the context of rational Cherednik algebras \cite{MR3519506}. If $A$ is a Drinfeld double, the one-dimensional rigid modules are characterized by \cite[Proposition 
10]{MR1220770} since these are the group-like elements of $A^*$. 

\begin{prop}
Let $\kuu_{(\cD(K),R)}(V)$ be as in \S\ref{subsub:double of bosonization} and $\lambda$ a one-dimensional weight of $\cD(K)$. Then $\lambda$ is a rigid $\kuu_{(\cD(K),R)}(V)$-module if 
and only if 
there exist group-like elements $g\in G(K)$ and $\eta\in G(K^*)$ such that $\lambda=\eta\ot g$ and $g\ot\eta$ is a central element of $\cD(K)$ acting trivially on $V$.

In particular, if $K$ is the group algebra of an abelian group, these are all the rigid modules of $\kuu_{(\cD(K),R)}(V)$. 
\end{prop}

\begin{proof}
As we said before, $\lambda$ is rigid if and only if $\lambda$ is a group-like element of $(\kuu_{(\cD(K),R)}(V))^*\simeq\left(\cD(\BV(V)\#K)\right)^*$. By \cite[Proposition 10]{MR1220770}, 
this is equivalent to the existence of $g\in G(\BV(V)\#K)$ and $\eta\in G((\BV(V)\#K)^*)$ with $\lambda=\eta\ot g$ and $g\ot\eta$ in the center of $\kuu_{(\cD(K),R)}(V)$. We first note that 
$G(\BV(V)\#K)=G(K)$ and $G((\BV(V)\#K)^*)=G(K^*)$. On the other hand, $g\ot\eta$ is central if and 
only if it is central in $\cD(K)$ and acts trivially on $V$ and on $\oV$. Since $g\ot\eta$ is a group-like element in $\kuu_{(\cD(K),R)}(V)$, $g\ot\eta$ acts trivially on $V$ 
if and only if it does on $\oV$. Hence the proposition follows. For the last part, we note that $\cD(K)$ is the group algebra of an abelian group if so is 
$K$, and consequently all its weights are one-dimensional.
\end{proof}


\end{document}